\documentclass[12pt]{article}
\usepackage{graphicx}
 \usepackage[reqno, namelimits, sumlimits]{amsmath}
   \usepackage{amsfonts}
   \usepackage{amssymb}
  \usepackage{amsthm}
\usepackage{xcolor}
\newtheorem{theorem}{Theorem}[section]

\newtheorem{proposition}[theorem]{Proposition}

\newtheorem{rem}[theorem]{Remark}

\def\N{{\mathbb N}}

\def\cA{{\mathcal A}}
\def\cE{{\mathcal E}}

\def\cH{{\mathcal H}}

\def\<{\langle}
\def\>{\rangle}

\date{ }
\begin{document}
\title{An extension of a Liapunov approach to the stabilization of second order coupled systems}
\maketitle

\begin{center}
Thierry Horsin\\
{\em Laboratoire M2N, EA7340\\
  CNAM, 292 rue Saint-Martin,\\
  75003 Paris\\
    France}\\
thierry.horsin@lecnam.net\\
\end{center}
\medskip
\begin{center}
Mohamed Ali JENDOUBI \\
{\em Universit\'e de Carthage,\\
Institut Pr\'eparatoire aux Etudes Scientifiques et Techniques,\\
B.P. 51 2070 La Marsa, Tunisia} \\
ma.jendoubi@fsb.rnu.tn \\
\end{center}
\begin{abstract}

\small{This paper deals with the convergence to $0$ of the energy of the solutions of a second order linear coupled system. In order to obtain the energy decay, we exhibit a Liapunov function.}

\vspace{6ex}
\noindent{\bf Mathematics Subject Classification 2010 (MSC2010):}
35B40, 49J15, 49J20.\vspace{6ex}

\noindent{\bf Key words:} damping, linear evolution equations,
dissipative hyperbolic equation, decay rates, Liapunov function.

\end{abstract}

\section{Introduction, functional framework.}

Let us consider a quite general coupled system in abstract form
 \begin{equation}\label{SystAbstraitIntro1}
\left\{ \begin{array}{ll}
u'' +Bu'+A_1 u +\alpha C v=0 & \\[2mm]
 v''+A_2 v+\alpha C^* u=0, &  
\end{array} \right. 
\end{equation}
 where $A_1$ and $A_2$ and $C$ are, in general, unbounded operators. F. Alabau and al. considered in \cite{ACK02},  the case when $C=Id$, and $A_1$ and $A_2$  are densely closed linear self-adjoint coercive operator and  $B$ is a coercive bounded self-adjoint operator. They proved that if $|\alpha| \Vert C\Vert <1$  then the energy of the solution $(u,v)$ in polynomially decreasing under quite large assumption on $A_1$ and $A_2$.  In this paper our main concern is the case $A_2=A_1^2$ which is a special case of the aforementioned paper. \\
 When $A_1=A_2$ and $|\alpha| \Vert C\Vert <1$, A. Haraux and M.A. Jendoubi proved in \cite{HJ2016} (see also \cite{HJ}) the polynomial convergence to $0$ of the energy  by means of a Liapunov method.\\
As we previously said, in this paper, we investigate such a method in the case when $A_2=A_1^2$ and $C=A_1^\beta $ with $\beta\in [0,\frac32]$. The main result of this paper is Theorem \ref{mainthm} which also proves the polynomial convergence to $0$ of the solution $(u,v)$. Compared to the result in \cite[p. 144, Proposition 5.3]{ACK02}, the convergence that we obtain is in weaker norms, but requires less regularity on the initial data.

 In order to motivate the Liapunov function that we construct in the proof of our main result, we explain the strategy in section \ref{scalarcase} in the framework of a coupled scalar differential system.

 In section \ref{sectionStronglyCoupled1} we introduce the functional framework and an existence theorem that lead to state and prove our main result, namely Theorem \ref{mainthm}.

\section{A Liapunov function for the scalar case\label{scalarcase}}
 
 As mentioned in the preceding section, we consider the (real) scalar coupled system

\begin{equation} \label{ODE}
\left\{ \begin{array}{ll}
u''+u'+\lambda u +cv=0 & \\[2mm]
v''+\mu v +cu=0 &  
\end{array} \right. \end{equation}
where $\lambda,\mu>0$,  and $c$ are such that $0< c^2<\lambda\mu$. The damping coefficient is set to $1$ for simplicity but a time scale change reduces general damping terms $bu'$ to this case. In order to shorten the formulas, let us introduce for each solution $(u,v)$ of \eqref{ODE}, its total energy
$$ \cE(u,u',v,v') = \frac12\left[   u'^2+   v'^2+\lambda u^2+\mu v^2\right]+ c uv .$$
Then we have for all $t\ge0$ $$ \frac{d}{dt} \cE(u,u',v,v') = -u'^2. $$
Now we introduce
$$K(t)=\frac12\left[   u'^2+   v'^2+\lambda u^2+\mu v^2\right].$$ Our first  result is the following

 \begin{proposition} \label{Liapscal} There are some constants $c>0$, $\delta>0$ such that
 $$\forall t\geq0\quad K(t)\leq ce^{-\delta t}K(0).$$
\end{proposition}
 \begin{proof}   For  all $\varepsilon>0$ we define the function
  \begin{equation}\label{FtLiapunovSystDiff}
H_\varepsilon=\cE-\varepsilon vv'+ 2\varepsilon uu'+\frac{3\varepsilon}{2c}(\mu u'v-\lambda uv').
\end{equation} 
It is easy to check that
\begin{equation}\label{estimationHepsilonK} 
C_{1} K(t)\leq H_{\varepsilon}(t)\leq C_{2}K(t)\end{equation}
where
$$C_{1}=\left[\frac{\sqrt{\lambda\mu}-\vert c\vert}{\sqrt{\lambda\mu}}-\varepsilon\left(\frac{2}{\min\left(\sqrt{\lambda},\sqrt{\mu}\right)}+\frac{3}{2\vert c\vert}\max\left(\sqrt{\lambda},\sqrt{\mu}\right)\right)\right]$$
and $$ C_{2}=\left[\frac{\sqrt{\lambda\mu}+\vert c\vert}{\sqrt{\lambda\mu}}+\varepsilon\left(\frac{2}{\min\left(\sqrt{\lambda},\sqrt{\mu}\right)}+\frac{3}{2\vert c\vert}\max\left(\sqrt{\lambda},\sqrt{\mu}\right)\right)\right].$$
Let $\varepsilon_{1}>0$ such that $C_{1}=0$ and $\varepsilon\in(0,\varepsilon_{1}).$ 
An obvious calculation gives
\begin{eqnarray*} 
H_\varepsilon'&=& -u'^2-\varepsilon v'^2-\varepsilon vv''+2\varepsilon u'^2+2\varepsilon uu''+\frac{3\varepsilon}{2c}[(\mu-\lambda) u'v'+\mu u''v-\lambda u v'' ]\\
&=&-(1-2\varepsilon)u'^2-\varepsilon v'^2+\varepsilon c uv+\varepsilon\mu v^2-2\varepsilon uu'-p\varepsilon\lambda u^2-2\varepsilon c uv\\
&&\quad +\frac{3\varepsilon}{2c}[(\mu-\lambda) u'v'-\mu u'v-\mu\lambda uv-\mu c v^2 +\lambda c u^2+\lambda \mu uv  ]\\
&=&-(1-2\varepsilon)u'^2-\varepsilon v'^2-\lambda \varepsilon\frac{1}{2}u^2-\mu \varepsilon\frac{1}{2}v^2-2\varepsilon uu'-\varepsilon c uv\\
&&\quad -\mu \varepsilon\frac{3}{2c}u'v+(\mu-\lambda) \varepsilon\frac{3}{2c}u'v'. \end{eqnarray*}
Now we have
\begin{eqnarray*} \lambda u^2+2c uv+\mu v^2 &=&\lambda\left(u^2+\frac{2c}{\lambda}uv+\frac{\mu}{\lambda}v^2\right)\\
&=&\lambda\left[\left(u+\frac{c}{\lambda}v\right)^2+\left(\frac{\mu}{\lambda}-\frac{c^2}{\lambda^2}\right)v^2\right]\\
&\geq& \frac{\lambda\mu-c^2}{\lambda} v^2.
\end{eqnarray*}
Similaraly we get
$$\lambda u^2+2cuv+\mu v^2\geq \frac{\lambda \mu-c^2}\mu u^2,$$
and then 
$$\lambda u^2+2c uv+\mu v^2\geq \frac{\lambda\mu-c^2}{2}\left(\frac{1}{\mu}u^2+\frac{1}{\lambda}v^2\right)$$ 
Using Young's inequality, we can find some constants $c_{1},c_{2},c_{3}>0$ such that
\begin{eqnarray*}&&\vert 2 u u'\vert\leq\frac{1}{8\mu}(\lambda\mu-c^2)u^2+c_{1}u'^2;\\
&& \frac{3}{2c}\vert\mu u'v\vert\leq\frac{1}{8\lambda}(\lambda\mu-c^2)v^2+c_{2}u'^2; \\
&& \frac{3}{2c}\vert\mu-\lambda\vert \vert u'v'\vert\leq\frac12 v'^2+c_{3}u'^2.
\end{eqnarray*}
Finally we obtain
$$H_{\varepsilon}'\leq-(1-(2+c_{1}+c_{2}+c_{3})\varepsilon)u'^2-\frac{\varepsilon}{2}v'^2-\frac{\varepsilon}{8\mu}(\lambda\mu-c^2)u^2-\frac{\varepsilon}{8\mu}(\lambda\mu-c^2)v^2$$
 
Now by choosing $\varepsilon\in(0,\varepsilon_{1})$  such that $1-(2+c_{1}+c_{2}+c_{3})\varepsilon>0,$ you can find some constant $C_{3}>0$ such that
$$H_{\varepsilon}'\leq -C_{3}K(t). $$
By combining this with the inequality \eqref{estimationHepsilonK}, we get for all $t\geq0$
$$H_{\varepsilon}'(t)\leq -\frac{C_{3}}{C_{2}}H_{\varepsilon}(t).$$
We conclude the proof by integrating this last inequality and using \eqref{estimationHepsilonK} again.
 \end{proof}

 \section {The case $A_2=A^2$,  and $C=A^\beta$ with $\beta \in [0,\frac32]$}\label{sectionStronglyCoupled1}

 This section is devoted to the proof of Theorem \ref{mainthm}. In order to proceed we first introduce the functional framework and give an exsitence theorem.

 \subsection{Functional framework}

Let  $H$ be a hilbert space, whose norm and scalar product will be denote
 $\Vert \cdot\Vert $ and $\langle\cdot,\cdot\rangle$ respectively. We consider $A:H\to H$ an unbounded closed self-adjoint operator such that the injection $D(A)\subset H$ is dense and compact. We assume moreover throughout the paper that there exists $a>0$ such that \begin{equation}\label{Apositif}
   \forall u\in D(A),\quad \langle Au,u\rangle\geq a\langle u,u\rangle.
   \end{equation}
 Following for example the exposition given in \cite{komlivre}, by denoting $(\lambda_n)_{n\in \N^*}$ the eigensequences of $A$, the largest $a$ for which \eqref{Apositif} is true is $\lambda_1$.

 Besides, let us consider $(e_n)_{n\in \N^*}$ an orthonormal basis of $H$ constituted by eigenvectors of $A$. For any $\beta>0$, we consider $u=\displaystyle \sum_{n=1}^\infty \langle u,e_{i}\rangle e_i\in H \quad (\iff \sum_{n=1}^\infty \langle u,e_{i}\rangle^2<\infty)$, and we define $A^\beta: H\to H$ by
 \begin{equation}
   A^\beta u=\sum_{i=1}^\infty \lambda_i^\beta \langle u,e_{i}\rangle e_i
 \end{equation}
 then (see e.g. \cite{komlivre})
 \begin{equation*}
   D(A^\beta)=\{u\in H,\, \sum_{i=1}^\infty \lambda_i^{2\beta}\langle u,e_{i}\rangle^2<\infty\}
 \end{equation*}
 and $A^\beta$ is an unbounded self-adjoint operator such that the inclusion $D(A^\beta)\subset H$ is dense and compact. We also have
 \begin{equation}\label{Apositif2}
   \forall u\in D(A^\beta ),\quad \langle A^\beta u,u\rangle\geq a\langle u,u\rangle.
 \end{equation}
 for some $a>0$. The largest $a$ for which this inequality is true being $\lambda_1^\beta$.

As usual we write $A^0=Id$. In this case of course the operator $A^\alpha$ is a continuous linear operator on $H$.
 
 We will denote $V=D(A^{1/2})$ and $W=D(A)$. Thus $V$ and $W$ are Hilbert spaces whose norms $\Vert\cdot\Vert_V$ and $\Vert\cdot\Vert_W$ are given respectively by
 \begin{equation*}
    \Vert u\Vert_{V}=\Vert A^{1/2}u\Vert,\quad\Vert u\Vert_{W}=\Vert Au\Vert.
   \end{equation*}
We  have, if we identify $H$ with its dual
 \begin{equation}
   W\subset V\subset H\subset V'\subset W'
 \end{equation}
 with dense and compact injections when the norms on the Hilbert spaces $V'$ and $W'$ are given by
 \begin{equation*}
   \forall u\in V',\quad \Vert u\Vert _{V'}=\langle u,A^{-1}u\rangle_{V',V}^{1/2}
   \end{equation*}
and \begin{equation*}
   \forall u\in W',\quad\Vert u\Vert _{W'}=\langle u,A^{-2}u\rangle_{W',W}^{1/2}
 \end{equation*}
where $\langle \cdot,\cdot\rangle_{V',V}$ denotes the action of $V'$ on $V$ (with a similar notation for $W$). Of course when $u\in H$ one has
 \begin{equation*}
   \Vert u\Vert _{V'}=\langle A^{-1}u,u\rangle^{1/2},\,\Vert u\Vert _{W'}=\langle A^{-2}u,u\rangle^{1/2}.
   \end{equation*}
Let us remark that with these definitions $A$ maps continuously $V$ to $V'$ and $A^2$ maps $W$ to $W'$.
\subsection{Existence result}
Let $\alpha$ and $\beta$ two reals numbers with $\beta\geq 0$. We recall that we consider the problem 
  \begin{equation}\label{SystOndeFortAbstrait1}
\left\{ \begin{array}{ll}
u'' +u'+A u +\alpha A^\beta v=0 & \\[2mm]
 v''+A^2 v+\alpha A^\beta u=0 &  
\end{array} \right. 
\end{equation}
which can be rewritten as the first order system

\begin{equation}\label{SystAbstrait}
\left\{ \begin{array}{ll}
u'-w=0 & \\[2mm]
v'-z=0 & \\[2mm]
w'+Au+w+\alpha A^\beta v=0 & \\[2mm]
z'+A^2v+\alpha A^\beta u=0. &  
\end{array} \right. 
\end{equation}
\noindent Let us  first establish an existence and uniqueness result for \eqref{SystOndeFortAbstrait1}.\\
We concentrate on the case $\beta\in [\frac12,\frac32]$, the case $\beta\in [0,\frac12)$ being easier.\\
Let us consider
$$\cH:= V\times W\times H\times H.$$
 For two elements of $\cH$ $U_i=(u_i,v_i,w_i,z_i)$, $i=1,2$, we define 
 \begin{eqnarray*}\langle  U_1,U_2\>_{\cH}&:=&\langle Au_2,u_1\>_{V',V}+\langle A^2v_2,v_1\>_{W',W}+\langle w_1,w_2\>_H+\langle z_1,z_2\>_{H}\\&&+\alpha\langle A^\beta v_2,u_1\>_{V',V}+\alpha\langle A^{\beta}u_2,v_1\>_{W',W},\end{eqnarray*} where $\langle \cdot,\cdot\>_{V',V}$ denotes the usual duality pairing between $V$ and $V'$, with similar notation for $W$, while $\langle \cdot,\cdot\>_H$ denotes the scalar product on $H$ for which it is an Hilbert space.

 It is straightforward to prove that $\langle \cdot,\cdot\>_{\cH}$ defines a scalar product on $\cH$ for which it is an Hilbert space.

We now consider the unbounded operator $\cA:\cH\longrightarrow \cH$ defined by
$$ D({\cal A}):=\{U=(u,v,w,z)\in \cH,\,  (-w,-z,Au+\alpha A^\beta v+w,A^2v+\alpha A^\beta u)\in \cH,\}$$
and for $U=(u,v,w,z)\in D(\cal A)$
$${\cal A}(U)=(-w,-z,Au+\alpha A^\beta v+w,A^2v+\alpha A^\beta u).$$
It is clear that $\cal A$ has a dense domain in $\cH$.\\
Let us remark that for any $U=(u,v,w,z)\in D(\cal A)$ one has
\begin{eqnarray*}
    \langle {\cal A}U,U\>_{\cH}&=&-\langle Au,w\>_{V',V}-\langle A^2v,z\>_{W',W}+\langle Au+\alpha A^\beta v+w,w\>_H+\langle A^2v+\alpha A^\beta u,z\>_H\\ &&-\alpha\langle A^\beta v,w\>_{V',V}-\alpha\langle A^{\beta}u,z\>_{W',W},
\end{eqnarray*}
and therefore $\langle {\cal A}U,U\>=\Vert w\Vert_H^2\geq 0$.
Indeed
$$\langle A^2v+\alpha A^\beta u,z\>_H=\langle A^2v+\alpha A^\beta u,z\>_{W',W},$$ since $A^\beta u\in W'$ for $\beta\in [1,3/2]$ and $u\in V=D(A^{1/2})$.\\
Let us show that $I+{\cal A}$ is onto. For this we take $(f,g,h,k)\in \cH$.
We want to find $(u,v,w,z)\in D({\cal A})$ such that
\begin{eqnarray*}
  &&u-w=f\\
  &&v-z=g\\
  &&Au+\alpha A^\beta v+2u=h+2f\\
  &&A^2+\alpha A^\beta u+v=k+g.
\end{eqnarray*}

We define $\Phi:(V\times W)^2\to {\mathbb R}$ by
\begin{eqnarray*}\Phi(u_1,v_1,u_2,v_2)&:=&\langle A^{1/2}u_1,A^{1/2}u_2\>_H+\langle Av_1,Av_2\>_H+
\alpha \langle A^\beta v_1,u_2\>_{V',V}+\\
&&\alpha \langle Au_1,A^{\beta-1}v_2\>_{V',V}+2\langle u_1,u_2\>_H+\langle v_1,v_2\>_H.
\end{eqnarray*}
Clearly  $\Phi$ is continuous on $V\times W$.
It is also clear that $\Phi$ is coercive if we assume $\vert \alpha\vert < \lambda_{1}^{\frac{3-2\beta}{2}}$.\\
By the Lax-Milgram theorem, there exists a unique $(u,v)\in V\times W$ such that
$$\forall (\delta u,\delta v)\in V\times W,\quad\Phi(u,v,\delta u,\delta v)=\langle h+2f,\delta u\>_H+\langle k+g,\delta v\>_H.$$
We therefore get  
$$A^2v+\alpha A^\beta u+v=g+k$$  $$Au+\alpha A^\beta v+2u=h+2f.$$
Now if we denote $w=u-f$ and $z=v-g$ then $w\in V$ since $u$ and $f$ do and $z\in W$ since $v$ and $g$ do.\\
We have thus proven that $\cA$ is maximal monotone. By classical theory, we get that 
\begin{theorem}\label{existhm}Assume that $\vert \alpha\vert \langle  \lambda_{1}^{\frac{3-2\beta}{2}}$. For any $(u_0,v_0,u_1,v_1)\in \cal H$, there exists a unique solution to \eqref{SystOndeFortAbstrait1} in $C(0,T,{\cal H})\times
  C^1(0,T,D({\cal A})')$.\end{theorem}

  \begin{rem}
    {\rm It is also well known that if $(u_0,v_0,u_1,v_1)\in D({\cal A})$ then the solution to \eqref{SystOndeFortAbstrait1} belongs to $C([0,T],D({\cal A}))\cap C^1(0,T,{\cal H})$.}
  \end{rem}
 \subsection{Main result of the paper}
Our main result is the following
 \begin{theorem}\label{mainthm}  Assume $\alpha \not=0$, $\beta\in[0,\frac32]$ and $\vert \alpha\vert < \lambda_{1}^{\frac{3-2\beta}{2}}$. Let $(u,v)$ be a solution of  \eqref{SystOndeFortAbstrait1}, then there exists a constant $c>0$  such that
\begin{eqnarray*} \forall t>0,  &&\Vert A^{\frac{\beta}{2}-1}u' (t)\Vert_{W'}^2+\Vert  A^{\frac{\beta}{2}-1} v'(t)\Vert_{W'}^2+\Vert{} A^{\frac{\beta-1}{2}}u(t)\Vert_{W'}^2+\Vert A^\frac{\beta}{2}v(t)\Vert_{W'}^2\\
 &&\le \frac{c} {t}\left[\Vert u' (0)\Vert^2+\Vert  v'(0)\Vert^2+\Vert{} u(0)\Vert_{V}^2+\Vert v(0)\Vert_{W}^2\right]\text{ if }\beta\in[0,1];\\ 
\forall t>0, &&\Vert A^{\frac{-\beta}{2}}u' (t)\Vert_{W'}^2+\Vert  A^{\frac{-\beta}{2}} v'(t)\Vert_{W'}^2+\Vert{} A^{\frac{1-\beta}{2}}u(t)\Vert_{W'}^2+\Vert A^{1-\frac{\beta}{2}}v(t)\Vert_{W'}^2\\&&\leq\frac{c}{t}\left[\Vert u' (0)\Vert^2+\Vert  v'(0)\Vert^2+\Vert{} u(0)\Vert_{V}^2+\Vert v(0)\Vert_{W}^2\right]\text{ if }\beta\in[1,\frac32].
 \end{eqnarray*} \end{theorem}
  \begin{rem}\label{reg}{\rm
    If we replace \eqref{SystOndeFortAbstrait1} by
    \begin{equation}\label{SystOndeFortAbstrait2}
\left\{ \begin{array}{ll}
u'' +Bu'+A u +\alpha A^\beta v=0 & \\[2mm]
 v''+A^2 v+\alpha A^\beta u=0 &  
\end{array} \right. 
    \end{equation}
    where, as mentionned in the introduction, $B$ is a bounded self-adjoint operator on $H$ for which there exists $\mu>0$ such that
    \begin{equation*}
      \langle Bu,u\rangle \geq \mu \Vert u\Vert ^2,
      \end{equation*}
 the results of Theorem \ref{existhm} and Remark \ref{reg}  remain true.}
\end{rem}
  \begin{rem}\label{pert}{\rm If we replace \eqref{SystOndeFortAbstrait1} by
    \begin{equation*}\label{SystOndeFortAbstrait3}
\left\{ \begin{array}{ll}
u'' +Bu'+A u +\alpha A^\beta v=0 & \\[2mm]
 v''+A_2 v+\alpha A^\beta u=0 &  
\end{array} \right. 
    \end{equation*}
    where $A_2$ is a self-adjoint unbounded operator such that $D(A_2)=D(A^2)$ and there exist $\nu_1,\,\nu_2>0$   such that
    \begin{equation*}
      \forall u\in D(A_2),\,\nu_1\langle A^2u,u\rangle \leq \langle A_2u,u\rangle\leq \nu_2\langle A^2u,u\rangle,
    \end{equation*}
    and if $B$ is as in the remark \ref{reg},
    the result of Theorem \ref{mainthm} remains true provided $|\alpha\vert$ is small enough (depending on $\lambda_1$, $\nu_1$ and $\nu_2$).}
  \end{rem}
\begin{rem} {\rm In the case $\beta=0$, in order to obtain the decay of the energy, we must assume 
$$A^2u(0)\in V,\quad A^2v(0)\in W,\quad A^2u'(0)\in H,\quad A^2v'(0)\in H.$$  
In the paper \cite{ACK02}, the authors obtain such a decay with merely
$$A u(0)\in V,\quad A^2v(0)\in W,\quad Au'(0)\in H,\quad A^2v'(0)\in H.$$
We, of course, would expect that the energy decay also holds with 
$$(u(0),v(0),u'(0),v'(0))\in D({\cal A})$$
but unfortunately we are not able to prove it for the moment being.}
\end{rem}  
 \begin{proof}[\bf Proof of Theorem \ref{mainthm}.]
   
   All the computations below will be made assuming that $(u_0,v_0,u_1,v_1)\in D({\cal A})$ which ascertains them. By density and continuity the inequalities stated in Theorem \ref{mainthm} remain true.
   
We introduce the energy of the system by
$$E(t)=\frac12\left[\Vert u' (t)\Vert^2+\Vert  v'(t)\Vert^2+\Vert{} u(t)\Vert_{V}^2+\Vert v(t)\Vert_{W}^2\right]+\alpha\langle  A^\beta v,u\>.$$
Then we have
$$E'(t)=-\Vert u'(t)\Vert^2.$$
Let $p>1$ and $\varepsilon>0$ two real numbers to be fixed later and let  
\begin{eqnarray*}H_{\varepsilon}=E-\varepsilon\lambda_{1}^{2-\beta}\langle A^{\beta-2}v,v'\>_{W'}+p\varepsilon\lambda_{1}^{-a}\langle A^{a}u,u'\>_{W'}+\rho\varepsilon\left[\langle u',v\>_{W'}-\langle u,A^{-1}v'\>_{W'}\right] 
\end{eqnarray*}
where $\rho=\frac{p+1}{2\alpha}\lambda_{1}^{2-\beta}$ and $a=\min(0,1-\beta)$. We find easily
\begin{eqnarray*} H_{\varepsilon}'&=&-\Vert u'\Vert^2-\varepsilon\lambda_{1}^{2-\beta}\Vert A^{\frac{\beta}{2}-1}v'\Vert_{W'}^2-\varepsilon\lambda_{1}^{2-\beta}\langle A^{\beta-2}v,v''\>_{W'}+p\varepsilon\lambda_{1}^{-a}\Vert A^{\frac{a}{2}}u'\Vert_{W'}^2\\
&&-p\varepsilon\lambda_{1}^{-a}\langle A^{a}u,u'+Au+\alpha A^\beta v\>_{W'} +\rho\varepsilon\langle u',v'\>_{W'}-\rho\varepsilon\langle u',A^{-1}v'\>_{W'}\\
&&-\rho\varepsilon\langle u'+Au+\alpha A^\beta v,v\>_{W'}+\rho\varepsilon\langle u,A^{-1}(A^2v+\alpha A^\beta u)\>_{W'} \\
&=&-\Vert u'\Vert^2-\varepsilon\lambda_{1}^{2-\beta}\Vert A^{\frac{\beta}{2}-1}v'\Vert_{W'}^2+\varepsilon\lambda_{1}^{2-\beta}\langle A^{\beta-2}v,A^2v+\alpha A^\beta u\>_{W'}+p\varepsilon\lambda_{1}^{-a}\Vert A^{\frac{a}{2}}u'\Vert_{W'}^2\\
&&-p\varepsilon\lambda_{1}^{-a}\langle A^{a}u,u'\>_{W'}-p\varepsilon\lambda_{1}^{-a}\Vert A^{\frac{a+1}{2}}u\Vert_{W'}^2-p\varepsilon\lambda_{1}^{-a}\alpha\langle A^{a}u,A^{\beta} v\>_{W'} +\rho\varepsilon\langle u',v'\>_{W'}\\
&&-\rho\varepsilon\langle u',A^{-1}v'\>_{W'}-\rho\varepsilon\langle u',v\>_{W'}-\rho\varepsilon\alpha\Vert A^{\frac{\beta}{2}}v\Vert_{W'} +\rho\varepsilon\alpha\Vert A^{\frac{\beta-1}{2}}u\Vert^2_{W'}\\
&=&-\Vert u'\Vert^2-\varepsilon\lambda_{1}^{2-\beta}\Vert A^{\frac{\beta}{2}-1}v'\Vert_{W'}^2+\varepsilon\lambda_{1}^{2-\beta}\Vert A^{\frac{\beta}{2}}v\Vert_{W'}^2+\varepsilon\lambda_{1}^{2-\beta}\alpha\langle A^{\beta-2}v,A^\beta u\>_{W'}+p\varepsilon\lambda_{1}^{-a}\Vert A^{\frac{a}{2}}u'\Vert_{W'}^2\\
&&-p\varepsilon\lambda_{1}^{-a}\langle A^{a}u,u'\>_{W'}-p\varepsilon\lambda_{1}^{-a}\Vert A^{\frac{a+1}{2}}u\Vert_{W'}^2-p\varepsilon\lambda_{1}^{-a}\alpha\langle A^{a}u,A^{\beta} v\>_{W'} +\rho\varepsilon\langle u',v'\>_{W'}\\
&&-\rho\varepsilon\langle u',A^{-1}v'\>_{W'}-\rho\varepsilon\langle u',v\>_{W'}-\rho\varepsilon\alpha\Vert A^{\frac{\beta}{2}}v\Vert_{W'}^2 +\rho\varepsilon\alpha\Vert A^{\frac{\beta-1}{2}}u\Vert^2_{W'}\\
&=&-\Vert u'\Vert^2+p\varepsilon\lambda_{1}^{-a}\Vert A^{\frac{a}{2}}u'\Vert_{W'}^2-\varepsilon\lambda_{1}^{2-\beta}\Vert A^{\frac{\beta}{2}-1}v'\Vert_{W'}^2-p\varepsilon\lambda_{1}^{-a}\Vert A^{\frac{a+1}{2}}u\Vert_{W'}^2-\varepsilon\frac{p-1}{2}\lambda_{1}^{2-\beta}\Vert A^{\frac{\beta}{2}}v\Vert_{W'}^2\\
&&+\rho\varepsilon\alpha\Vert A^{\frac{\beta-1}{2}}u\Vert^2_{W'}+\varepsilon\lambda_{1}^{2-\beta}\alpha\langle A^{\beta-2}v,A^\beta u\>_{W'}-p\varepsilon\lambda_{1}^{-a}\langle A^{a}u,u'\>_{W'}-p\varepsilon\lambda_{1}^{-a}\alpha\langle A^{a}u,A^{\beta} v\>_{W'}\\
&& +\rho\varepsilon\langle u',v'\>_{W'}-\rho\varepsilon\langle u',A^{-1}v'\>_{W'}-\rho\varepsilon\langle u',v\>_{W'}. 
\end{eqnarray*}
\noindent {\bf First case :} $\beta\in [0,1]$. In this case $a=0$. We have
\begin{eqnarray*} H_{\varepsilon}'
&=&-\Vert u'\Vert^2+p\varepsilon \Vert  u'\Vert_{W'}^2-\varepsilon\lambda_{1}^{2-\beta}\Vert A^{\frac{\beta}{2}-1}v'\Vert_{W'}^2-p\varepsilon \Vert A^{\frac{1}{2}}u\Vert_{W'}^2-\varepsilon\frac{p-1}{2}\lambda_{1}^{2-\beta}\Vert A^{\frac{\beta}{2}}v\Vert_{W'}^2\\
&&+\rho\varepsilon\alpha\Vert A^{\frac{\beta-1}{2}}u\Vert^2_{W'}+\varepsilon\lambda_{1}^{2-\beta}\alpha\langle A^{\beta-2}v,A^\beta u\>_{W'}-p\varepsilon \langle u,u'\>_{W'}-p\varepsilon \alpha\langle u,A^{\beta} v\>_{W'}\\
&& +\rho\varepsilon\langle u',v'\>_{W'}-\rho\varepsilon\langle u',A^{-1}v'\>_{W'}-\rho\varepsilon\langle u',v\>_{W'}. 
\end{eqnarray*}
Now since
$$\Vert A^{\frac{\beta-1}{2}}u\Vert^2_{W'}\leq \frac{1}{\lambda_{1}^{2-\beta}}\langle Au,u\>_{W'}=\frac{1}{\lambda_{1}^{2-\beta}}\Vert A^{\frac{1}{2}}u\Vert_{W'}^2,$$ we get
\begin{eqnarray*} H_{\varepsilon}'&\leq&-\Vert u'\Vert^2+p\varepsilon\Vert  u'\Vert_{W'}^2-\varepsilon\lambda_{1}^{2-\beta}\Vert A^{\frac{\beta}{2}-1}v'\Vert_{W'}^2-\frac{p-1}{2}\varepsilon \Vert A^{\frac{1}{2}}u\Vert_{W'}^2-\varepsilon\frac{p-1}{2}\lambda_{1}^{2-\beta}\Vert A^{\frac{\beta}{2}}v\Vert_{W'}^2\\
&&+\varepsilon\lambda_{1}^{2-\beta}\alpha\langle A^{\beta-2}v,A^\beta u\>_{W'}-p\varepsilon \langle u,u'\>_{W'}-p\varepsilon \alpha\langle u,A^{\beta} v\>_{W'}\\
&& +\rho\varepsilon\langle u',v'\>_{W'}-\rho\varepsilon\langle u',A^{-1}v'\>_{W'}-\rho\varepsilon\langle u',v\>_{W'}.
\end{eqnarray*}
Let us remark that
\begin{eqnarray*}
\varepsilon\lambda_{1}^{2-\beta}\alpha\langle A^{\beta-2}v,A^\beta u\>_{W'}
&\leq &\varepsilon\lambda_{1}^{2-\beta}\vert\alpha\vert\Vert A^{\frac{\beta}{2}}v\Vert_{W'}\Vert A^{\frac{3\beta}{2}-2}u\Vert_{W'} \\
&\leq &\varepsilon\lambda_{1}^{2-\beta}\vert\alpha\vert\Vert A^{\frac{\beta}{2}}v\Vert_{W'}\frac{1}{\lambda_{1}^{\frac{5-3\beta}{2}}}\Vert A^{\frac{1}{2}}u\Vert_{W'} \\
&\leq &\varepsilon\lambda_{1}^{\frac{\beta-1}{2}}\vert\alpha\vert\Vert A^{\frac{\beta}{2}}v\Vert_{W'} \Vert A^{\frac{1}{2}}u\Vert_{W'}
\end{eqnarray*}
and that
\begin{eqnarray*}
-p\varepsilon \alpha\langle  u,A^{\beta} v\>_{W'}&\leq&p\varepsilon \vert\alpha\vert\Vert A^{\frac{\beta}{2}}v\Vert_{W'} \Vert A^{\frac{\beta}{2}}u\Vert_{W'}\\
&\leq&p\varepsilon \lambda_{1}^{\frac{\beta-1}{2}}\vert\alpha\vert\Vert A^{\frac{\beta}{2}}v\Vert_{W'} \Vert A^{\frac{1}{2}}u\Vert_{W'}.
\end{eqnarray*}
Thus
\begin{eqnarray*}
&&\varepsilon\lambda_{1}^{2-\beta}\alpha\langle A^{\beta-2}v,A^\beta u\>_{W'}-p\varepsilon \alpha\langle  u,A^{\beta} v\>_{W'}\\
&\leq &\varepsilon\lambda_{1}^{\frac{\beta-1}{2}}\vert\alpha\vert\Vert A^{\frac{\beta}{2}}v\Vert_{W'} \Vert A^{\frac{1}{2}}u\Vert_{W'}+p\varepsilon \lambda_{1}^{\frac{\beta-1}{2}}\vert\alpha\vert\Vert A^{\frac{\beta}{2}}v\Vert_{W'} \Vert A^{\frac{1}{2}}u\Vert_{W'}\\
&\leq &\varepsilon\lambda_{1}^{\frac{\beta-1}{2}}\vert\alpha\vert(p +1)\Vert A^{\frac{\beta}{2}}v\Vert_{W'} \Vert  A^{\frac{1}{2}}u\Vert_{W'}\\
&\leq &\varepsilon\lambda_{1}^{\frac{\beta-1}{2}} (p +1) \frac{\vert\alpha\vert}{2} \left(\gamma\Vert A^{\frac{1}{2}}u\Vert_{W'}^2+\frac{1}{\gamma}\Vert A^{\frac{\beta}{2}}v\Vert_{W'}^2\right)\\
\end{eqnarray*}
where we choose $\gamma>0$ such that
\begin{equation}\label{Definitiondeltazeta}\delta:=\frac{p-1}{2}-\lambda_{1}^{\frac{\beta-1}{2}} (p +1 )\frac{\vert\alpha\vert}{2}\gamma>0,\quad \zeta:=\frac{p-1}{2}\lambda_{1}^{2-\beta}-\lambda_{1}^{\frac{\beta-1}{2}} (p +1 )\frac{\vert\alpha\vert}{2\gamma}>0\end{equation}
which is equivalent to 
$$\frac{\lambda_{1}^{\frac{\beta-1}{2}} (p +1 )\vert\alpha\vert}{(p-1)\lambda_{1}^{2-\beta}}<\gamma<\frac{p-1}{\lambda_{1}^{\frac{\beta-1}{2}} (p +1 )\vert\alpha\vert}.$$
This choice is possible provided that
\begin{equation}\label{equAA2Beta}\left(\frac{p +1}{p-1}\right)^2<\frac{\lambda_{1}^{3-2\beta}}{\vert\alpha\vert^2}.\end{equation}
Now we choose $p>1$ such that \eqref{equAA2Beta} is satisfied. Then we have
\begin{eqnarray*} H_{\varepsilon}'&=&-\Vert u'\Vert^2+p\varepsilon\Vert  u'\Vert_{W'}^2-\varepsilon\lambda_{1}^{2-\beta}\Vert A^{\frac{\beta}{2}-1}v'\Vert_{W'}^2- \varepsilon\delta \Vert A^{\frac{1}{2}}u\Vert_{W'}^2-\varepsilon\zeta\Vert A^{\frac{\beta}{2}}v\Vert_{W'}^2\\
&&-p\varepsilon \langle  u,u'\>_{W'} +\rho\varepsilon\langle u',v'\>_{W'}-\rho\varepsilon \langle u',A^{-1}v'\>_{W'}-\rho\varepsilon\langle u',v\>_{W'}.
\end{eqnarray*}
Using Young's inequality, we find some constants $c_{1},c_{2},c_{3},c_{4}>0$ such that
\begin{eqnarray*}
&&-p\langle  u,u'\>_{W'}\leq c_{1}\Vert u'\Vert^2+\frac{\delta}{2}\Vert A^{\frac{1}{2}}u\Vert_{W'}^2;\quad(\delta\text{ as in }\eqref{Definitiondeltazeta})\\ 
&&\rho\langle u',v'\>_{W'}\leq c_{2}\Vert u'\Vert^2+ \frac{\lambda_{1}^{2-\beta}}{3}\Vert A^{\frac{\beta}{2}-1}v'\Vert_{W'}^2\\
&&\rho\langle u',A^{-1}v'\>_{W'}\leq c_{3}\Vert u'\Vert^2+ \frac{\lambda_{1}^{2-\beta}}{3}\Vert A^{\frac{\beta}{2}-1}v'\Vert_{W'}^2\\
&&-\rho\langle u',v\>_{W'}\leq c_{4}\Vert u'\Vert^2+\frac{\zeta}{2}\Vert A^{\frac{\beta}{2}}v\Vert_{W'}^2\quad(\zeta\text{ as in }\eqref{Definitiondeltazeta}).\end{eqnarray*}
By choosing $\varepsilon$ small enough, we find  a constant $\eta = \eta (p, \varepsilon) >0$ such that for all $t\ge 0$
\begin{eqnarray*} H_{\varepsilon}'\leq-\eta\left(\Vert u'\Vert^2+\Vert A^{\frac{\beta}{2}-1}v'\Vert_{W'}^2+ \Vert A^{\frac{1}{2}}u\Vert_{W'}^2+\Vert A^{\frac{\beta}{2}}v\Vert_{W'}^2\right)
\end{eqnarray*}
Let $$\tilde{E}=\frac12\left[\Vert A^{\frac{\beta}{2}-1}u' (t)\Vert_{W'}^2+\Vert  A^{\frac{\beta}{2}-1} v'(t)\Vert_{W'}^2+\Vert{} A^{\frac{\beta-1}{2}}u(t)\Vert_{W'}^2+\Vert A^\frac{\beta}{2}v(t)\Vert_{W'}^2\right]$$$$+\alpha\langle A^{2\beta-2}v,u\>_{W'}$$ 
and $$K(t)=\Vert A^{\frac{\beta}{2}-1}u' (t)\Vert_{W'}^2+\Vert  A^{\frac{\beta}{2}-1} v'(t)\Vert_{W'}^2+\Vert{} A^{\frac{\beta-1}{2}}u(t)\Vert_{W'}^2+\Vert A^\frac{\beta}{2}v(t)\Vert_{W'}^2.$$
For all $t\geq 0$, we have $$\tilde{E}'=-\Vert A^{\frac{\beta}{2}-1}u' (t)\Vert_{W'}^2.$$
Then $\tilde{E}$ is nonincreasing. Observe that
$$\left\vert\alpha\langle A^{2\beta-2}v,u\>_{W'}\right\vert\leq\frac{\vert\alpha\vert}{\lambda_{1}^{\frac{3-2\beta}{2}}}\Vert{} A^{\frac{\beta-1}{2}}u(t)\Vert_{W'}\Vert A^\frac{\beta}{2}v(t)\Vert_{W'},$$
from which we deduce that
\begin{equation}\label{energieKNouveau2}\frac{{\lambda_{1}^{\frac{3-2\beta}{2}}}-\vert \alpha\vert}{2\lambda_{1}^{\frac{3-2\beta}{2}}}K(t)\leq \tilde{E}\leq \frac{\lambda_{1}^{\frac{3-2\beta}{2}}+\vert\alpha\vert}{2\lambda_{1}^{\frac{3-2\beta}{2}}}K(t).\end{equation}
Now since
$$\Vert A^{\frac{\beta}{2}-1}u' (t)\Vert_{W'}\leq \frac{1}{\lambda_{1}^{2-\frac{\beta}{2}}}\Vert  u'\Vert,\qquad \Vert{} A^{\frac{\beta-1}{2}}u(t)\Vert_{W'}\leq\frac{1}{\lambda_{1}^{1-\frac{\beta}{2}}}\Vert A^{\frac{1}{2}}u\Vert_{W'},$$ 
then there exists a  constant $\gamma>0$ such that for all $t\ge 0$
\begin{equation}\label{InegH'K-13} H'_\varepsilon\leq -\gamma K(u,v,u',v') . \end{equation}
 From  \eqref{InegH'K-13}, assuming $\varepsilon $ possibly smaller in order to achieve positivity of the quadratic form $H_{\varepsilon}$, we get
$$\int_0^tK(u(s),v(s),u'(s),v'(s))\,ds\leq\frac{1}{\gamma}H_\varepsilon(u(0),v(0),u'(0),v'(0)).$$
Using  inequality \eqref{energieKNouveau2}, we obtain
$$\frac{2\lambda_{1}^{\frac{3-2\beta}{2}}}{\lambda_{1}^{\frac{3-2\beta}{2}}+\vert\alpha\vert} \int_0^t  \tilde{E}(u(s),v(s),u'(s),v'(s))  \,ds\leq \frac{1}{\gamma}H_\varepsilon(u(0),v(0),u'(0),v'(0)).$$
Now since $\tilde{E}$ is nonincreasing, it follows 
$$\tilde{E}(u(t),v(t),u'(t),v'(t)) \leq  \frac{\lambda_{1}^{\frac{3-2\beta}{2}}+\vert\alpha\vert}{2\lambda_{1}^{\frac{3-2\beta}{2}}\gamma}\frac{1}{t}H_\varepsilon(u(0),v(0),u'(0),v'(0)).$$
Using  inequality \eqref{energieKNouveau2} we get
$$K(u(t),v(t),u'(t),v'(t)) \leq \frac{\lambda_{1}^{\frac{3-2\beta}{2}}+\vert\alpha\vert}{(\lambda_{1}^{\frac{3-2\beta}{2}}-\vert\alpha\vert)\gamma}\frac{1}{t}H_\varepsilon(u(0),v(0),u'(0),v'(0)).$$
 \noindent {\bf Second case :} $\beta\in (1,\frac32]$. In this case $a=1-\beta$.
\begin{eqnarray*} H_{\varepsilon}'
&=&-\Vert u'\Vert^2+p\varepsilon\lambda_{1}^{\beta-1}\Vert A^{\frac{1-\beta}{2}}u'\Vert_{W'}^2-\varepsilon\lambda_{1}^{2-\beta}\Vert A^{\frac{\beta}{2}-1}v'\Vert_{W'}^2-p\varepsilon\lambda_{1}^{\beta-1}\Vert A^{1-\frac{\beta}{2}}u\Vert_{W'}^2-\varepsilon\frac{p-1}{2}\lambda_{1}^{2-\beta}\Vert A^{\frac{\beta}{2}}v\Vert_{W'}^2\\
&&+\rho\varepsilon\alpha\Vert A^{\frac{\beta-1}{2}}u\Vert^2_{W'}+\varepsilon\lambda_{1}^{2-\beta}\alpha\langle A^{\beta-2}v,A^\beta u\>_{W'}-p\varepsilon\lambda_{1}^{\beta-1}\langle A^{1-\beta}u,u'\>\\
&&-p\varepsilon\lambda_{1}^{\beta-1}\alpha\langle A^{1-\beta}u,A^{\beta} v\>_{W'} +\rho\varepsilon\langle u',v'\>_{W'}-\rho\varepsilon\langle u',A^{-1}v'\>_{W'}-\rho\varepsilon\langle u',v\>_{W'}. 
\end{eqnarray*}
Now since
$$\Vert A^{\frac{\beta-1}{2}}u\Vert^2_{W'}\leq \frac{1}{\lambda_{1}^{3-2\beta}}\langle A^{2-\beta}u,u\>_{W'}=\frac{1}{\lambda_{1}^{3-2\beta}}\Vert A^{1-\frac{\beta}{2}}u\Vert_{W'}^2,$$
we get
\begin{eqnarray*} H_{\varepsilon}'&\leq&-\Vert u'\Vert^2+p\varepsilon\lambda_{1}^{\beta-1}\Vert A^{\frac{1-\beta}{2}}u'\Vert_{W'}^2-\varepsilon\lambda_{1}^{2-\beta}\Vert A^{\frac{\beta}{2}-1}v'\Vert_{W'}^2-\frac{p-1}{2}\varepsilon\lambda_{1}^{\beta-1}\Vert A^{1-\frac{\beta}{2}}u\Vert_{W'}^2\\
&&-\varepsilon\frac{p-1}{2}\lambda_{1}^{2-\beta}\Vert A^{\frac{\beta}{2}}v\Vert_{W'}^2 +\varepsilon\lambda_{1}^{2-\beta}\alpha\langle A^{\beta-2}v,A^\beta u\>_{W'}-p\varepsilon\lambda_{1}^{\beta-1}\langle A^{1-\beta}u,u'\>\\
&&-p\varepsilon\lambda_{1}^{\beta-1}\alpha\langle A^{1-\beta}u,A^{\beta} v\>_{W'} +\rho\varepsilon\langle u',v'\>_{W'}-\rho\varepsilon\langle u',A^{-1}v'\>_{W'}-\rho\varepsilon\langle u',v\>_{W'}
\end{eqnarray*}
Let us remark that
\begin{eqnarray*}
\varepsilon\lambda_{1}^{2-\beta}\alpha\langle A^{\beta-2}v,A^\beta u\>_{W'}
&\leq &\varepsilon\lambda_{1}^{2-\beta}\vert\alpha\vert\Vert A^{\frac{\beta}{2}}v\Vert_{W'}\Vert A^{\frac{3\beta}{2}-2}u\Vert_{W'} \\
&\leq &\varepsilon\lambda_{1}^{2-\beta}\vert\alpha\vert\Vert A^{\frac{\beta}{2}}v\Vert_{W'}\frac{1}{\lambda_{1}^{3-2\beta}} \Vert A^{1-\frac{\beta}{2}}u\Vert_{W'} \\
&\leq &\varepsilon\lambda_{1}^{\beta-1}\vert\alpha\vert\Vert A^{\frac{\beta}{2}}v\Vert_{W'}  \Vert A^{1-\frac{\beta}{2}}u\Vert_{W'}
\end{eqnarray*}
and that
$$-p\varepsilon\lambda_{1}^{\beta-1}\alpha\langle A^{1-\beta}u,A^{\beta} v\>_{W'} \leq p\varepsilon\lambda_{1}^{\beta-1} \vert\alpha\vert\Vert A^{\frac{\beta}{2}}v\Vert_{W'} \Vert A^{1-\frac{\beta}{2}}u\Vert_{W'},$$
we therefore get
\begin{eqnarray*}
&&\varepsilon\lambda_{1}^{2-\beta}\alpha\langle A^{\beta-2}v,A^\beta u\>_{W'}-p\varepsilon\lambda_{1}^{\beta-1}\alpha\langle A^{1-\beta}u,A^{\beta} v\>_{W'}\\
&\leq &\varepsilon\lambda_{1}^{\beta-1}\vert\alpha\vert\Vert A^{\frac{\beta}{2}}v\Vert_{W'}  \Vert A^{1-\frac{\beta}{2}}u\Vert_{W'}+p\varepsilon\lambda_{1}^{\beta-1} \vert\alpha\vert\Vert A^{\frac{\beta}{2}}v\Vert_{W'} \Vert A^{1-\frac{\beta}{2}}u\Vert_{W'}\\
&=&\varepsilon(p+1)\lambda_{1}^{\beta-1}\vert\alpha\vert\Vert A^{\frac{\beta}{2}}v\Vert_{W'}  \Vert A^{1-\frac{\beta}{2}}u\Vert_{W'}\\
&\leq &\varepsilon\lambda_{1}^{\beta-1}(p +1)\frac{\vert\alpha\vert}{2} \left(\gamma\Vert A^{\frac{1-\beta}{2}}u\Vert_{W'}^2+\frac{1}{\gamma}\Vert A^{\frac{\beta}{2}}v\Vert_{W'}^2\right)\\
\end{eqnarray*}
where we choose $\gamma>0$ such that
$$\delta:=\frac{p-1}{2}\lambda_{1}^{\beta-1}-\lambda_{1}^{\beta-1} (p +1 )\frac{\vert\alpha\vert}{2}\gamma>0,\quad \zeta:=\frac{p-1}{2}\lambda_{1}^{2-\beta}-\lambda_{1}^{\beta-1} (p +1 )\frac{\vert\alpha\vert}{2\gamma}>0$$
which is equivalent to
$$\frac{\lambda_{1}^{\beta-1} (p +1 )\vert\alpha\vert}{(p-1)\lambda_{1}^{2-\beta}}<\gamma<\frac{p-1}{  (p +1 )\vert\alpha\vert}.$$
This choice is possible provided that
\begin{equation}\label{equAA2Beta32}\left(\frac{p +1}{p-1}\right)^2<\frac{\lambda_{1}^{3-2\beta}}{\vert\alpha\vert^2}.\end{equation}
We choose $p>1$ such that \eqref{equAA2Beta32} is satisfied. Then we have
\begin{eqnarray*} H_{\varepsilon}'&=&-\Vert u'\Vert^2+p\varepsilon\lambda_{1}^{\beta-1}\Vert A^{\frac{1-\beta}{2}}u'\Vert_{W'}^2-\varepsilon\lambda_{1}^{2-\beta}\Vert A^{\frac{\beta}{2}-1}v'\Vert_{W'}^2- \varepsilon\delta \Vert A^{1-\frac{\beta}{2}}u\Vert_{W'}^2-\varepsilon\zeta\Vert A^{\frac{\beta}{2}}v\Vert_{W'}^2\\
&&-p\varepsilon\lambda_{1}^{\beta-1}\langle A^{1-\beta}u,u'\> +\rho\varepsilon\langle u',v'\>_{W'}-\rho\varepsilon\langle u',A^{-1}v'\>_{W'}-\rho\varepsilon\langle u',v\>_{W'}
.
\end{eqnarray*}
There are $c_{1},c_{2},c_{3},c_{4}>0$ such that
\begin{eqnarray*}
&&-p\lambda_{1}^{\beta-1}\langle A^{1-\beta}u,u'\>\leq c_{1}\Vert u'\Vert^2+\frac{\delta}{2}\Vert A^{1-\frac{\beta}{2}}u\Vert_{W'}^2,\\ 
&&\rho\langle u',v'\>_{W'}\leq c_{2}\Vert u'\Vert^2+ \frac{\lambda_{1}^{2-\beta}}{3}\Vert A^{\frac{\beta}{2}-1}v'\Vert_{W'}^2,\\
&&\rho\varepsilon\langle u',A^{-1}v'\>_{W'}\leq c_{3}\Vert u'\Vert^2+ \frac{\lambda_{1}^{2-\beta}}{3}\Vert A^{\frac{\beta}{2}-1}v'\Vert_{W'}^2,\\
  &&-\rho\langle u',v\rangle _{W'}\leq c_{4}\Vert u'\Vert^2+\frac{\zeta}{2}\Vert A^{\frac{\beta}{2}}v\Vert_{W'}^2.\end{eqnarray*}
By choosing $\varepsilon$ small enough, we find  a constant $\eta = \eta (p, \varepsilon) >0$ such that for all $t\ge 0$
\begin{eqnarray*} H_{\varepsilon}'\leq-\eta\left(\Vert u'\Vert^2+\Vert A^{\frac{\beta}{2}-1}v'\Vert_{W'}^2+ \Vert A^{1-\frac{\beta}{2}}u\Vert_{W'}^2+\Vert A^{\frac{\beta}{2}}v\Vert_{W'}^2\right)
\end{eqnarray*}
Let $$\tilde{E}=\frac12\left[\Vert A^{\frac{-\beta}{2}}u' (t)\Vert_{W'}^2+\Vert  A^{\frac{-\beta}{2}} v'(t)\Vert_{W'}^2+\Vert{} A^{\frac{1-\beta}{2}}u(t)\Vert_{W'}^2+\Vert A^{1-\frac{\beta}{2}}v(t)\Vert_{W'}^2\right]+\alpha\langle v,u\>_{W'}$$
and 
 $$K(t)=\Vert A^{\frac{-\beta}{2}}u' (t)\Vert_{W'}^2+\Vert  A^{\frac{-\beta}{2}} v'(t)\Vert_{W'}^2+\Vert{} A^{\frac{1-\beta}{2}}u(t)\Vert_{W'}^2+\Vert A^{1-\frac{\beta}{2}}v(t)\Vert_{W'}^2.$$
For all $t\geq 0$, we have
$$\tilde{E}'=-\Vert A^{-\frac{\beta}{2}}u' (t)\Vert_{W'}^2.$$
Then $\tilde{E}$ is nonincreasing.\\
Since $$\left\vert\alpha\langle  v,u\>_{W'}\right\vert\leq\frac{\vert\alpha\vert}{\lambda_{1}^{\frac{3-2\beta}{2}}}\Vert{} A^{\frac{1-\beta}{2}}u(t)\Vert_{W'}\Vert A^{1-\frac{\beta}{2}}v(t)\Vert_{W'}$$
then we get
\begin{equation}\label{energieKNouveau1}\frac{\lambda_{1}^{\frac{3-2\beta}{2}}-\vert \alpha\vert}{2\lambda_{1}^{\frac{3-2\beta}{2}}}K(t)\leq \tilde{E}\leq \frac{\lambda_{1}^{\frac{3-2\beta}{2}}+\vert\alpha\vert}{2\lambda_{1}^{\frac{3-2\beta}{2}}}K(t).\end{equation}
Now since
$$\Vert A^{-\frac{\beta}{2}}v' (t)\Vert_{W'}\leq \frac{1}{\lambda_{1}^{\beta-1}}\Vert A^{\frac{\beta}{2}-1}v'\Vert_{W'},\qquad \Vert{} A^{\frac{1-\beta}{2}}u(t)\Vert_{W'}\leq\frac{1}{\lambda_{1}^{\frac{1}{2}}}\Vert A^{1-\frac{\beta}{2}}u\Vert_{W'},$$
then there exists a  constant $\gamma>0$ such that for all $t\ge 0$
\begin{equation}\label{InegH'K-12} H'_\varepsilon\leq -\gamma K(u,v,u',v') . \end{equation} From  \eqref{InegH'K-12}, assuming $\varepsilon $ possibly smaller in order to achieve positivity of the quadratic form $H_{\varepsilon}$, we get
$$\int_0^tK(u(s),v(s),u'(s),v'(s))\,ds\leq\frac{1}{\gamma}H_\varepsilon(u(0),v(0),u'(0),v'(0)).$$
Using  inequality \eqref{energieKNouveau1}, we obtain
$$\frac{2\lambda_{1}^{\frac{3-2\beta}{2}}}{\lambda_{1}^{\frac{3-2\beta}{2}}+\vert\alpha\vert} \int_0^t  \tilde{E}(u(s),v(s),u'(s),v'(s))  \,ds\leq \frac{1}{\gamma}H_\varepsilon(u(0),v(0),u'(0),v'(0)).$$
Now since $\tilde{E}$ is nonincreasing, it follows 
$$\tilde{E}(u(t),v(t),u'(t),v'(t)) \leq  \frac{\lambda_{1}^{\frac{3-2\beta}{2}}+\vert\alpha\vert}{2\lambda_{1}^{\frac{3-2\beta}{2}}\gamma}\frac{1}{t}H_\varepsilon(u(0),v(0),u'(0),v'(0)).$$
Using  inequality \eqref{energieKNouveau1} we get
$$K(u(t),v(t),u'(t),v'(t)) \leq \frac{\lambda_{1}^{\frac{3-2\beta}{2}}+\vert\alpha\vert}{(\lambda_{1}^{\frac{3-2\beta}{2}}-\vert\alpha\vert)\gamma}\frac{1}{t}H_\varepsilon(u(0),v(0),u'(0),v'(0)).$$
\end{proof} 
\section{Examples }
This section is devoted to giving examples of operators to which Theorem \ref{mainthm} applies.

\noindent {\bf Example 1} The first case that we consider is when $H=L^2(\Omega)$ and
\begin{equation*}
Au=-\sum_{i,j=1}^N\frac{\partial}{\partial x_i}(a_{ij}\frac{\partial}{\partial x_j}u),
\end{equation*}
  where the coefficients $a_{i,j}\in C^1(\bar{\Omega})$ satisfy
  \begin{eqnarray}
    a_{ij}=a_{ji},\,\forall i,j,
    \end{eqnarray}
  and the matrix $(a_{i,j})$ is uniformly coercive on $\overline{\Omega}$. With these assumptions we have $D(A)=H^2(\Omega)\cap H_0^1(\Omega)$.

\noindent {\bf Example 2} The second example considered is $H=L^2(\Omega)$, $D(A)=\{u\in H^2(\Omega),\,\frac{\partial u}{\partial n}=0\mbox{ on }\partial \Omega\},$
  and
\begin{equation*}
Au=-\Delta u+\rho_1 u,
\end{equation*}
  where $\rho_1>0$.\\
  Here we can, as in \cite{ACK02} consider the case where
  \begin{equation*}
    D(A_2)=D(A^2),\, A_2u=\Delta^2u+\rho_2u,
  \end{equation*}
  where $\rho_2>0$.

\noindent {\bf Example 3} Let us remark that due to remark \ref{pert} and the Poincar\'e inequality, our result applies to the case when $A_{1}=A$ is as in example 1 and  $A_2=A_1^2+\zeta A_1$ for any $\zeta>0$.

\noindent {\bf Acknowledgements: } The second author wishes to thank the department of mathematics and statistics and the laboratory M2N of the CNAM where this work has been initiated. The first author wishes to thank the Tunisian Mathematical Society (SMT) for its kind invitation to its annual congress during which this work has been completed.
\bibliographystyle{amsplain}
\providecommand{\bysame}{\leavevmode\hbox to3em{\hrulefill}\thinspace}

\end{document}